\documentclass[11pt,reqno]{amsart}
\usepackage{a4wide}
\usepackage[utf8x]{inputenc}
\usepackage[T1]{fontenc}
\usepackage[english]{babel}
\usepackage{hyperref}
\usepackage{amsfonts,amsmath,amssymb,amsthm,amsrefs}
\usepackage{latexsym}
\usepackage{layout}
\usepackage{dsfont}
\usepackage{color}

\usepackage{graphicx}

\usepackage{ifpdf}
\usepackage{color}
\definecolor{webgreen}{rgb}{0,.5,0}
\definecolor{webbrown}{rgb}{.8,0,0}
\definecolor{emphcolor}{rgb}{0.5,0.95,0.95}

\newtheorem{theorem}{Theorem}[section]
\newtheorem{proposition}{Proposition}[section]
\newtheorem{lemma}{Lemma}[section]
\newtheorem{remark}{Remark}[section]
\newtheorem{definition}{Definition}[section]
\newtheorem{assumption}{Assumption}[section]
\newcommand{\p}{\mathbb{P}}
\newcommand{\e}{\mathbb{E}}

\newcommand{\reals}{\mathbb{R}}

\newcommand{\ind}{\mathds{1}}
\newcommand{\Ind}{\mathds{1}}

\newcommand{\diff}{\mathrm{d}}

\newcommand{\cH}{\mathcal{H}}

\newcommand{\homega}{\omega_q}

\newcommand{\ruintime}{\sigma^\pi_\omega}
\newcommand{\ruinbarrier}{\sigma^b_\omega}

\DeclareMathOperator*{\argmin}{arg\,min}

\allowdisplaybreaks

\begin{document}

\title[Optimal dividends with a level-dependent intensity of ruin]{Optimality of a barrier strategy in a spectrally negative L\'evy model with a level-dependent intensity of bankruptcy}

\author[D. Mata]{Dante Mata}

\author[J.-F. Renaud]{Jean-Fran\c{c}ois Renaud}
\address{D\'epartement de math\'ematiques, Universit\'e du Qu\'ebec \`a Montr\'eal (UQAM), 201 av.\ Pr\'esident-Kennedy, Montr\'eal (Qu\'ebec) H2X 3Y7, Canada}
\email{mata\_lopez.dante@uqam.ca, renaud.jf@uqam.ca}

\date{\today}

\keywords{Stochastic control, optimal dividends, spectrally negative L\'{e}vy processes, Omega models, Parisian ruin, barrier strategies.}

\begin{abstract}
We consider de Finetti's stochastic control problem for a spectrally negative L\'evy process in an Omega model. In such a model, the (controlled) process is allowed to spend time under the critical level but is then subject to a level-dependent intensity of bankruptcy. First, before considering the control problem, we derive some analytical properties of the corresponding Omega scale functions. Second, we prove that exists a barrier strategy that is optimal for this control problem under a mild assumption on the L\'evy measure. Finally, we analyse numerically the impact of the bankruptcy rate function on the optimal strategy.
\end{abstract}

\maketitle


\section{Introduction}

The stochastic control problem concerned with the maximization of dividend payments in a model based on a (general) spectrally negative Lévy process (SNLP) has attracted a lot of research interest since the papers of Avram, Palmowski \& Pistorius \cite{avram-palmowski-pistorius_2007} and Loeffen \cite{loeffen_2008}. In that problem, a dividend strategy is said to be optimal if it maximises the expected present value of dividend payments made up to the ruin time, which is a standard first-passage time. The recent developments in Parisian fluctuation theory for SNLPs have inspired the use of a Parisian ruin time as the termination time in that control problem; see, e.g., \cites{czarna-palmowski_2014, renaud2019, xu-et-al_2021, locas-renaud_2024b}. In particular, when using an exponential Parisian ruin time, i.e., with exponentially distributed delays, it has been possible to formulate a sufficient condition, namely on the Lévy measure, under which one can guarantee that a barrier strategy is optimal; see \cite{renaud2019}. This sufficient condition, for the optimality of a barrier strategy, dates back to \cite{loeffen-renaud_2010} in the classical version of the problem. It was proved that this condition on the Lévy measure guarantees the convexity of $W_q^\prime$, which is the derivative of the $q$-scale function of the corresponding SNLP; here, $q$ is the time-preference factor in the maximization problem. This convexity property is needed in particular to characterize the optimal barrier level. Then, in \cite{renaud2019}, it was proved that this convexity property was transferable to another scale function, namely $Z_q(\cdot, \Phi(p+q))$, which is at the core of the problem when Parisian ruin at rate $p$ is used as the termination time.

It is now well known that a Parisian first-passage time with exponential delays has connections with the distribution of certain occupation times of the underlying process; see, e.g., \cite{landriault-et-al_2011}. In that paper, there is also a connection with so-called \textit{Omega models}, which were introduced in \cites{albrecher-et-al_2011b, gerber-et-al_2012}. In such a model, and in the spirit of Parisian ruin aiming at making a distinction between default and bankruptcy, the probability of bankruptcy is a function of the surplus level. More precisely, given a \textit{bankruptcy rate function} $\omega(\cdot)$, if the surplus is at a level $x<0$, then the corresponding rate of bankruptcy is given by $\omega(x)$; a full definition is provided in Definition~\ref{def:bankruptcy-rate-function}. This concept of level-dependent bankruptcy was made rigorous for SNLPs by Li \& Palmowski \cite{li-palmowski_2018}. Indeed, they provided a general level-dependent notion of discounting for SNLPs, in the spirit of the connection between exponential Parisian ruin and occupation times. In fact, exponential Parisian ruin at rate $p$ corresponds to the \textit{bankruptcy rate function} $\omega(x)=p \ind_{(-\infty,0)}(x)$. In \cite{li-palmowski_2018}, fluctuation identities with level-dependent discounting are expressed in terms of some generalised scale functions, also called \textit{Omega scale functions}. These functions are defined implicitly as the solutions of Volterra integral equations; see Equation~\eqref{eq:functional-eq} below.

\subsection{Stochastic model}

In this paper, we study the problem of dividend maximization for a SNLP. Henceforth, consider a filtered probability space $\left( \Omega, \mathcal{F}, \left\lbrace \mathcal{F}_t, t \geq 0 \right\rbrace, \p \right)$ such that $\mathcal{F}_\infty := \bigvee_{t \geq 0} \mathcal{F}_t \subsetneq \mathcal{F}$, where $\bigvee_{t \geq 0} \mathcal{F}_t$ is the smallest $\sigma$-algebra containing $\mathcal{F}_t$ for all $t \geq 0$. Let $X=\left\lbrace X_t , t \geq 0 \right\rbrace$ be a SNLP (adapted to the filtration $\left\lbrace \mathcal{F}_t, t \geq 0 \right\rbrace$) with Laplace exponent
\[
\psi(\theta) = \gamma \theta + \frac{1}{2} \sigma^2 \theta^2 + \int^{\infty}_0 \left( \mathrm{e}^{-\theta z} - 1 + \theta z \ind_{(0,1]}(z) \right) \nu(\mathrm{d}z) , \quad \theta \geq 0 ,
\]
where $\gamma \in \reals$ and $\sigma \geq 0$, and where $\nu$ is a $\sigma$-finite measure on $(0,\infty)$, called the L\'{e}vy measure of $X$, satisfying
\begin{equation*}
\int^{\infty}_0 (1 \wedge x^2) \nu(\mathrm{d}x) < \infty .
\end{equation*}
The corresponding scale functions $\left\lbrace W_p , p \geq 0 \right\rbrace$ are then given by
\begin{equation*}
\int_0^\infty \mathrm{e}^{-\theta x} W_p (x) \mathrm{d}x = \left(\psi(\theta)-p \right)^{-1} , 
\end{equation*}
for all $\theta> \Phi(p)=\sup \left\lbrace \lambda \geq 0 \colon \psi(\lambda)=p \right\rbrace$. For more details on SNLPs and scale functions, see, e.g., \cites{kuznetsov-et-al_2012,kyprianou_2014}.

In what follows, we will use the following notation: the law of $X$ when starting from $X_0 = x$ is denoted by $\p_x$ and the corresponding expectation by $\e_x$. We write $\p$ and $\e$ when $x=0$. Finally, we exclude the case of $X$ having decreasing paths.

\subsection{Optimisation problem}\label{sect:optimisation-problem}

Let $X$ be the underlying wealth (cash, surplus) process. A control dividend strategy $\pi$ is represented by a non-decreasing, right-continuous and adapted stochastic process $L^\pi = \left\lbrace L^\pi_t , t \geq 0 \right\rbrace$ such that $L^\pi_{0-} = 0$, where $L^\pi_t$ represents the cumulative amount of dividends made/paid up to time $t$ under this control strategy. For a given strategy $\pi$, the corresponding controlled process $U^\pi = \left\lbrace U^\pi_t , t \geq 0 \right\rbrace$ is defined by $U^\pi_t = X_t - L^\pi_t$. A strategy is said to be admissible if, for all $t$, we have $0 \leq L_t^\pi - L_{t-}^\pi \leq U^\pi_{t-}$, and if $\int_{(0,\infty)} \ind_{\lbrace U^{\pi}_t < 0 \rbrace} \diff L^\pi_t = 0$. Let $\Pi$ be the set of admissible strategies.

In order to define the termination/bankruptcy time of interest, let us fix a \textit{bankruptcy rate function} $\omega \colon \reals \to [0,\infty)$, that is roughly speaking a non-negative and non-increasing function having support in $(-\infty,0]$; a full definition is provided in Definition~\ref{def:bankruptcy-rate-function}. Correspondingly, we define the termination time associated with $\omega(\cdot)$: for $\pi \in \Pi$, set
\begin{equation}\label{omega-ruin}
\ruintime = \inf \left\lbrace t>0 \colon \int_0^t \omega(U^\pi_s) \mathrm{d}s > \mathbf{e}_1 \right\rbrace ,
\end{equation}
where $\mathbf{e}_1$ is an independent random variable following the exponential distribution with unit mean.

Finally, fix a time-preference/discount rate $q > 0$. The performance function associated to an arbitrary $\pi \in \Pi$ is defined by
\[
v_\pi (x) = \e_x \left[ \int_0^{\ruintime} \mathrm{e}^{-q t} \mathrm{d}L^\pi_t \right] , \quad x \in \reals .
\]
The goal is to compute
\[
v_\ast (x) = \sup_{\pi \in \Pi} v_\pi (x) , \quad x \in \reals ,
\]
the (optimal) value function of this control problem by finding an optimal strategy $\pi_\ast \in \Pi$, i.e., an admissible strategy such that
\[
v_{\pi_\ast} (x) = v_\ast (x) , \quad x \in \reals .
\]
More precisely, in this paper, we find a sufficient condition on the Lévy measure for a barrier strategy (see Section~\ref{sect:barrier-strategies}) to be optimal, in the spirit of the existing literature on the maximization of dividends in a SNLP model that has followed since \cite{loeffen_2008}. See Theorem~\ref{thm:main-result} below.

\subsection{Contributions and organisation of the paper}

As alluded to above, our problem is a generalization of the one studied in \cite{renaud2019}, which was itself an \textit{asymptotic generalization} of the classical problem studied in \cites{loeffen_2008, loeffen-renaud_2010}. It turns out that no extra condition on the model is required to guarantee that a barrier strategy is optimal, even with this more general definition of bankruptcy; compare Theorem~\ref{thm:main-result} below with Theorem~1.1 in \cite{loeffen-renaud_2010} and Theorem~1 in \cite{renaud2019}.

To {solve this optimisation problem}, we study Omega scale functions beyond what is done in \cite{li-palmowski_2018}. More precisely, we study differentiability and convexity properties of those scale functions, in the spirit of what is known for classical scale functions (see Theorem~\ref{thm:differentiability}); these analytical properties are needed in the verification step of our solution to this control problem. It must be pointed out that these results about Omega scale functions are of independent interest and improve upon the theory developed in \cite{li-palmowski_2018}. The same methodology would apply to the other scale functions studied in that other paper, for example to $\mathcal{W}^\omega$.

The rest of this paper is structured as follows. In the next section, the definitions of a discounting intensity and then of a bankruptcy rate function are given, and a detailed analysis of Omega scale functions is provided. In Section~\ref{Sec:Optim:Dividend}, we come back to our optimisation problem: first, we compute the performance function of an arbitrary barrier strategy, and then we prove that one of these strategies is optimal for our control problem. Section~\ref{Sec:Verification} contains the proof of Theorem~\ref{thm:main-result}, which provides a solution to the optimisation problem. The last section contains numerical experiments illustrating the sensitivity of the optimal strategy, and the value function, with respect to the bankruptcy rate function.

\section{Scale functions}\label{sect:scale}

In what follows, it is assumed that, for an integrable function $f$, if $a \geq b$ then $\int_a^b f(x) \diff x = 0$. In general, by $\int_a^b$ we mean $\int_{[a,b)}$.

For the rest of this paper, we make the following standing assumption:
\begin{assumption}\label{standing-assumption}
If the paths of $X$ are of bounded variation (BV), then the Lévy measure is assumed to be absolutely continuous with respect to Lebesgue measure. 
\end{assumption}
As a consequence of this assumption, we have that, for any $p \geq 0$, the scale function $W_p$ is continuously differentiable on $(0,\infty)$; see, e.g., \cite{kuznetsov-et-al_2012}. {For future use, let us recall that
\begin{equation*}\label{initialvalues}
\begin{split}
W_p(0+)= &
\begin{cases}
1/(\gamma+\int_0^1 z \nu(\mathrm{d}z)) & \text{when $\sigma=0$ and $\int_0^1 z \nu(\mathrm{d}z)<\infty$},  \\
0 & \text{otherwise},
\end{cases}\\ 
W_p^{\prime}(0+)=&
\begin{cases}
2/\sigma^2 & \text{when $\sigma>0$,} \\
(\nu(0,\infty)+p)/(\gamma+\int_0^1 z \nu(\mathrm{d}z))^2 & \text{when $\sigma=0$ and $\nu(0,\infty)<\infty$,} \\
\infty & \text{otherwise.}
\end{cases}
\end{split}
\end{equation*}
For more details on scale functions, see, e.g., \cites{kuznetsov-et-al_2012}.
}

Recall that $q$ is the discount rate. Since \cites{avram-et-al_2007,loeffen_2008}, it has been widely known that the $q$-scale function $W_q$ is an essential object in the study of the optimal dividends problem with classical ruin. When exponential Parisian ruin is considered, as in \cite{renaud2019}, a second family of scale functions is needed. If Parisian ruin occurs at rate $\phi>0$, i.e., if $\omega(x) = \phi \ind_{(-\infty,0)}(x)$, then $Z_q (\cdot;\Phi(\phi+q))$ is the scale function of interest. In general, for $r,s>0$, such a function is defined, for each $x \in \reals$, by
\begin{multline}\label{eq:Zq-def}
Z_r (x;\Phi(s)) = \mathrm{e}^{\Phi(s)x} \left(1-(s-r) \int_0^x \mathrm{e}^{-\Phi(s)y} W_{r}(y) \mathrm{d}y \right) \\
= (s-r) \int_0^\infty \mathrm{e}^{-\Phi(s)y} W_{r}(x+y) \mathrm{d}y ,
\end{multline}
where, for $x \leq 0$, we have $Z_r (x;\Phi(s))=\mathrm{e}^{\Phi(s) x}$; note that the last expression in~\eqref{eq:Zq-def} is valid when $r<s$. In particular, we have $Z_q (x;\Phi(\phi+q))= \phi \int_0^\infty \mathrm{e}^{-\Phi(\phi+q)y} W_q(x+y) \mathrm{d}y$.

It is known (see, e.g., \cite{renaud2019}) that $x \mapsto Z_r(x;\Phi(s))$ is continuous on $\reals$ and that:
\begin{enumerate}
\item If $X$ has paths of BV, then $x \mapsto Z_r(x;\Phi(s))$ is continuously differentiable on $\reals \setminus \{0\}$ with
\begin{equation}\label{eq:derivative-at-zero-Z}
Z_r^\prime(0+;\Phi(s)) = \Phi(s) - (s-r) W_r(0+) \quad \text{and} \quad Z_r^\prime(0-;\Phi(s)) = \Phi(s) .
\end{equation}
\item If $X$ has paths of UBV, then $x \mapsto Z_r(x;\Phi(s))$ is continuously differentiable on $\reals$.
\end{enumerate}
If $\sigma>0$, then $x \mapsto Z_r(x;\Phi(s))$ is twice continuously differentiable on $\reals \setminus \{0\}$.

Before defining the scale function of interest in our context, let us give two identities involving the scale functions defined so far.

\begin{lemma}\label{lemma:identities}
For real numbers $a \leq b < x$, we have
	\begin{equation}\label{eq:identity-W}
		\int_a^x W_{p}(x-y) W_{r}(y-a) \mathrm{d}y = \frac{W_{r}(x-a) - W_{p}(x-a)}{r-p}
	\end{equation}
	and, for $s > \max\{p,r\}$, we have
	\begin{equation}\label{eq:identity-Z}
		\int_a^x W_{p}(x-y) Z_r (y-a;\Phi(s)) \mathrm{d}y = \frac{Z_r (x-a;\Phi(s)) - Z_p (x-a;\Phi(s))}{r-p} .
	\end{equation}
\end{lemma}

To the best of our knowledge, a version of the identity in~\eqref{eq:identity-W} first appeared in Lemma~3 of \cite{pistorius_2004} in which its proof is based on a well-known power series expansion of the $p$-scale function $W_p$. Another proof, based on Laplace transforms, is provided in \cite{loeffen-et-al_2014}. The identity in~\eqref{eq:identity-Z} first appeared in Lemma~4.1 of \cite{albrecher-et-al_2016}, in a slightly modified version; its proof is based on Laplace transforms. In Appendix~\ref{A:identities}, we provide an arguably simpler and shorter proof.

\subsection{Discounting intensity}

In the next section, we will define and study a generalised scale function associated to a given \textit{discounting intensity}. Now, let us provide a full definition of the latter.

\begin{definition}\label{def:bankruptcy-rate-function}
We say that $\omega \colon \reals \to [0,\infty)$ is a discounting intensity if:
\begin{enumerate}
\item it is non-increasing;
\item it is \textit{ultimately constant}, i.e., there exist $a \in (-\infty,0]$ and $\phi > 0$ for which $\omega(x)=\phi$ for all $x \in (-\infty,a)$, and there exists $\rho \in [0, \omega(0-)]$ such that $\omega(x)=\rho$ for all $x \in [0,\infty)$;
\item it is piecewise absolutely continuous, i.e., there exists a finite partition $\{a_1, \dots, a_{n+1}\}$ of $[a,0]$ such that $a = a_1 < \cdots < a_n \leq a_{n+1} = 0$ and for which $\omega$ is absolutely continuous on each sub-interval $(a_k,a_{k+1})$, for $k \in \lbrace 1,\cdots,n \rbrace$.
\end{enumerate}
In particular, in an Omega model, we say that $\omega$ is a bankruptcy rate function if it is a discounting intensity with $\rho=0$.
\end{definition}

Note that the assumption of being ultimately constant comes from Section~2.4 in \cite{li-palmowski_2018}. This assumption has been relaxed in \cite{vidmar2019}. However, in what follows, our methodology relies heavily on this assumption.


As opposed to other elements in the partition, a discounting intensity might or might not have jumps at $a$ and $0$. More precisely, we have $\phi=\omega(a-) \geq \omega(a+)$ and $\omega(0-) \geq \rho=\omega(0)$. Thus, if $\omega$ is continuous on $(-\infty,0)$, then $n=1$ and $a \leq 0$ and thus the only possible discontinuity point is now at $0$. This is the case for example if $\omega(x) = \phi \ind_{(-\infty,0)} (x) + \rho \ind_{[0,\infty)}(x)$. However, it is possible for $\omega$ to be continuous on $\reals$. For example,
\[
\omega(x) = \phi \ind_{(-\infty,-\phi)}(x) - x \ind_{[-\phi,0)}(x)
\]
is a continuous discounting intensity/bankruptcy rate function.

An important family of discounting intensities/bankruptcy rate functions is given by piecewise constant functions. For example, for a partition given by $-\infty=a_0 < a_1 < \dots < a_n < a_{n+1}=0$ and decreasing rates given by $\phi = p_0 > p_1 > \dots > p_n > 0$, we can define
\begin{equation}\label{omega-step-function}
\omega(x) = \sum_{i=1}^{n+1} p_{i-1} \ind_{[a_{i-1},a_i)}(x) = p_0 \ind_{(-\infty,a_1)}(x) + p_1 \ind_{[a_1,a_2)}(x) + \dots + p_n \ind_{[a_n, a_{n+1})}(x) ,
\end{equation}
where, for simplicity, we used the notation $[a_0,a_1) := (-\infty,a_1)$. Intuitively speaking, the function given in~\eqref{omega-step-function} is such that the rate $p_i$ is applied when the process is in the interval $[a_i,a_{i+1})$. As the process goes deeper into the \textit{red zone}, since we have $p_i < p_{i-1}$, the discounting rates increase.

\begin{remark}\label{rem:discontinuous}
In what follows, for the sake of readability, many statements will be made using the set $\{a_1,\dots,a_{n+1}\}$ even though, when $\omega$ is continuous at $a_{n+1}=0$, the statement remains true with $\{a_1,\dots,a_{n}\}$ instead.
\end{remark}

\begin{remark}
In an Omega model, a piecewise constant bankruptcy rate function yields a termination time that is already significantly more general than a Parisian ruin time with exponential delays. Indeed, Parisian ruin with exponential rate $\phi>0$ corresponds to the bankruptcy rate function $\omega(x) = \phi \ind_{(-\infty,0)}(x)$.
\end{remark}

\subsection{Omega scale functions}\label{Sec:Solution:Complete}

For the control problem under study, we need a generalization of both $W_{q}(\cdot)$ and $Z_q (\cdot;\Phi(\phi+q))$. A theory of scale functions based on a discounting intensity $\omega$ has been developed in \cite{li-palmowski_2018}. In that paper, the $\omega$-scale function $\cH^{\omega} \colon \reals \to \reals$ is defined as the solution of the following functional equation:
\begin{equation}\label{eq:functional-eq}
\cH^{\omega}(x) = \mathrm e^{\Phi(\phi)(x-a)} + \int_{a}^x W_{\phi}(x-y) (\omega(y) - \phi) \cH^{\omega}(y) \diff y , \quad x \in \reals .
\end{equation}
Also, as it is known, from Lemma~2.1 in \cite{li-palmowski_2018}, that $\cH^\omega$ is a nonnegative and locally bounded function, {then from~\eqref{eq:functional-eq} we further have that it is continuous}. In Theorem~\ref{thm:differentiability}, we will discuss its smoothness.
%

\begin{remark}
Note that~\eqref{eq:functional-eq} does not appear verbatim in \cite{li-palmowski_2018}. We have performed a translation in Equations~(2.22)-(2.23) of \cite{li-palmowski_2018} to obtain~\eqref{eq:functional-eq}. This modification is needed to be in agreement with our definition of a discounting intensity.
\end{remark}

First, let us provide a continuum of alternate functional equations for the definition of the $\omega$-scale function $\cH^{\omega}$.

\begin{lemma}\label{lem:alternate-functional-equations}
For a fixed $p \geq 0$, we have
	\begin{equation}\label{omega:Volterra}
		\cH^{\omega}(x) = Z_p(x-a;\Phi(\phi)) + \int_{a}^x W_p(x-y) (\omega(y)-p) \cH^{\omega}(y) \diff y , \quad x \in \reals .
	\end{equation}
\end{lemma}

\begin{proof}
Using the definition of $\cH^{\omega}$ as given by~\eqref{eq:functional-eq}, we can write
\begin{multline*}
\int_a^x W_p(x-y) \cH^{\omega} (y) \diff y \\
= \int_a^x \mathrm{e}^{\Phi(\phi)(y-a)} W_p(x-y) \diff y  + \int_a^x W_p(x-y) \left[ \int_a^y W_\phi(y-z) (\omega(z)-\phi) \cH^{\omega} (z) \diff z \right] \diff y \\
= \int_0^{x-a} \mathrm{e}^{\Phi(\phi)(x-y-a)} W_p(y) \diff y  + \int_a^x (\omega(z)-\phi) \cH^{\omega} (z) \left[ \int_z^x W_p(x-y) W_\phi(y-z) \diff y \right] \diff z ,
\end{multline*}
in which, by Lemma~\eqref{lemma:identities}, we have
\[
\int_z^x W_p(x-y) W_\phi(y-z) \diff y = \frac{W_p(x-z) - W_\phi(x-z)}{p-\phi} .
\]
Using~\eqref{eq:functional-eq} to replace $\int_{a}^x W_{\phi}(x-y) (\omega(y) - \phi) \cH^{\omega}(y) \diff y$ by $\cH^{\omega}(x) - \mathrm e^{\Phi(\phi)(x-a)}$ and identifying $Z_p(x-a;\Phi(\phi))$, we can further write
\begin{multline*}
(p- \phi) \int_a^x W_p(x-y) \cH^{\omega} (y) \diff y \\
= Z_p(x-a;\Phi(\phi)) - \cH^{\omega} (x) + \int_{a}^x W_p(x-y) (\omega(y)-\phi) \cH^{\omega}(y) \diff y .
\end{multline*}
The result follows.
\end{proof}

Note that, in the last lemma, if $p=\phi$ then $Z_p(y;\Phi(\phi)) = \mathrm e^{\Phi(\phi)y}$ and we recover~\eqref{eq:functional-eq}.

\begin{remark}
Of course, $\cH^{\omega}$ is independent of the arbitrarily chosen value of $p$. Later, when solving the control problem, we will make an appropriate choice for $p$.
\end{remark}

We are ready to state our first main result. Recall that $a_1=a$ and $a_{n+1}=0$.

\begin{theorem}\label{thm:differentiability}
Under Assumption~\ref{standing-assumption}, the $\omega$-scale function $\cH^{\omega}$ has the following differentiability properties:
{
\begin{enumerate}
\item If $\sigma^2>0$, then $\cH^{\omega}$ is continuously differentiable on $\reals$.
\item If $\sigma^2=0$ and $X$ has paths of UBV, with the additional assumption that the first derivative of the scale function is convex on $(0,\infty)$, then $\cH^{\omega}$ is continuously differentiable on $\reals$.
\item If $\sigma^2=0$ and $X$ has paths of BV with $\nu(0,\infty)=\infty$, with the additional assumption that the first derivative of the scale function is convex on $(0,\infty)$, then $\cH^{\omega}$ is continuously differentiable on $\reals \setminus \{a_1, \dots, a_{n+1}\}$.
\item If $\sigma^2=0$ and $\nu(0,\infty)<\infty$, then $\cH^{\omega}$ is continuously differentiable on $\reals \setminus \{a_1, \dots, a_{n+1}\}$
\end{enumerate}
Furthermore, if $\sigma >0$ and if the first derivative of the scale function is convex on $(0,\infty)$, then $\cH^{\omega}$ is twice continuously differentiable on $\reals \setminus \{a_1, \dots, a_{n+1}\}$.
}
\end{theorem}

Note that the differentiability properties of $\cH^{\omega}$ are reminiscent of those for classical scale functions; see, e.g., \cite{kuznetsov-et-al_2012}. {We also remark that if $\omega$ is continuous at $a_{n+1}=0$, then the statement remains valid with $\lbrace a_1,\cdots,a_n \rbrace$ instead (see Remark~\ref{rem:discontinuous}).}




\subsection{Proof of Theorem~\ref{thm:differentiability}}

The proof of Theorem~\ref{thm:differentiability} is based on the following lemma.

\begin{lemma}\label{lemma:diff-volterra}
Let $a \in \reals$ and let $H \colon \reals \to \reals$ be piecewise continuous on $(a,\infty)$, with discontinuities on $\{a_1, \dots, a_{n+1}\}$. Also, let {$g \colon (0,\infty) \to \reals$} be such that:
\begin{enumerate}
    \item $|g(0+)|<\infty$;
    \item {$g$ is continuously differentiable};
    \item $g^\prime$ satisfies one of the following conditions:
    \begin{enumerate}
        \item $|g^\prime(0+)|<\infty$;
        \item it is non-negative and convex;
        \item it is non-decreasing;
    \end{enumerate}
    \item $\int_0^A |g^\prime(y)| \diff y < \infty$, for any $A>0$.
\end{enumerate}
Under these assumptions, the function $f(x) = \int_a^x g(x-y) H(y) \diff y$ is {continuously} differentiable on $(a,\infty) \setminus \{a_1, \dots, a_{n+1}\}$ with
\[
f^\prime(x) = g(0+) H(x) + \int_a^x g^\prime(x-y) H(y) \diff y .
\]
{Furthermore, if $g(0+)=0$ then $f$ is continuously differentiable on $(a,\infty)$.}
\end{lemma}

\begin{proof}
{First, note that for each $b \in (a,\infty) \setminus \{a_1, \dots, a_{n+1}\}$, there exists a constant $C_b > 0$ such that $|H(y)|\leq C_b$ for all $y \in [a,b]$.}

Let us fix $x \in (a,\infty) \setminus \{a_1, \dots, a_{n+1}\}$. For $h>0$, we can write
\begin{equation}\label{eq:differential-quotient}
\frac{f(x+h)-f(x)}{h} = \int_{a}^x \frac{g(x+h-y)-g(x-y)}{h} H(y) \diff y + \frac{1}{h} \int_x^{x+h} g(x+h-y) H(y) \diff y .
\end{equation}
Before computing the limit when $h$ decreases to zero, let us further assume that $h \in (0,1)$. Since $g^\prime$ is continuous, we can write, for $y \in (a,x)$,
\begin{multline*}
\left| \frac{g(x+h-y)-g(x-y)}{h} H(y) \right| \leq C_x \left| \frac{g(x+h-y)-g(x-y)}{h} \right| \\
\leq \frac{C_x}{h} \int_{x-y}^{x+h-y} |g^\prime (u)| \diff u \leq C_x \max_{u \in [x-y,x-y+1]} |g^\prime (u)| .
\end{multline*}
Note that, under Assumption~(3a), by continuity we have, for all $y \in (a,x)$,
\[
\max_{u \in [x-y,x-y+1]} |g^\prime (u)| \leq \max_{u \in [0,x-a+1]} |g^\prime (u)| < \infty ,
\]
in which {$g^\prime(0):=g^\prime(0+)$}. Otherwise, under Assumption~(3b)-(3c), we can further write
\[
\max_{u \in [x-y,x-y+1]} |g^\prime (u)| \leq \max \left\lbrace |g^\prime (x-y)|, |g^\prime (x-y+1)| \right\rbrace \leq |g^\prime (x-y)| + |g^\prime (x-y+1)| .
\]
Note that, under Assumption~(4), $y \mapsto |g^\prime (x-y)| + |g^\prime (x-y+1)|$ is integrable on $(a,x)$.

Now, let us compute the limits in~\eqref{eq:differential-quotient}. For the first integral, using the Dominated Convergence Theorem, we have
\[
\lim_{h \downarrow 0} \int_{a}^x \frac{g(x+h-y)-g(x-y)}{h} H(y) \diff y = \int_{a}^x g^\prime(x-y) H(y) \diff y .
\]
For the second integral, since $x \notin \{a_1, \dots, a_{n+1}\}$, by continuity, for a given $\epsilon>0$, there exists $\delta>0$ such that: for any $u,v \in (0,\delta)$,
\[
\left| g(u) H(x+v) - g(0+) H(x) \right| < \epsilon . 
\]
Hence, if $h \in (0,\delta)$, then we also have that $x+h-y$ lies in $(0,\delta)$ for each $y \in (x,x+h)$ and we can write
\begin{multline*}
\left| \frac{1}{h} \int_x^{x+h} g(x+h-y) H(y) \diff y - g(0+) H(x) \right| \\
{\leq} \frac{1}{h} \int_x^{x+h} \left| g(x+h-y) H(y) - g(0+) H(x) \right| \diff y < \epsilon .
\end{multline*}
This proves that
\[
\lim_{h \downarrow 0} \frac{1}{h} \int_x^{x+h} g(x+h-y) H(y) \diff y = g(0+) H(x) .
\]

Similarly, for $h>0$, we can write
\[
\frac{f(x-h)-f(x)}{-h} = \int_{a}^{x-h} \frac{g(x-y)-g(x-h-y)}{h} H(y) \diff y \\
+ \frac{1}{h} \int_{x-h}^x g(x-y) H(y) \diff y .
\]
As above, we can show that
\[
\lim_{h \downarrow 0} \frac{f(x-h)-f(x)}{-h} = g(0+) H(x) + \int_{a}^{x} g^\prime (x-y) H(y) \diff y .
\]
The result follows.
%
%
\end{proof}

Recall that $\cH^{\omega}$ is continuous on $\reals$ and, as an increasing function, $\cH^{\omega}$ is differentiable almost everywhere on $\reals$; in fact, it admits right-hand and left-hand derivatives everywhere on $\reals$.

Recall from Equation~\eqref{omega:Volterra} in Lemma~\ref{lem:alternate-functional-equations} that, for any $p \geq 0$,
\[
\cH^{\omega}(x) = Z_p(x-a;\Phi(\phi)) + \int_{a}^x W_p(x-y) (\omega(y)-p) \cH^{\omega}(y) \diff y .
\]
The differentiability properties of $x \mapsto Z_p(x-a;\Phi(\phi))$ have already been discussed in Section~\ref{sect:scale}. Consequently, to prove Theorem~\ref{thm:differentiability}, we must now analyse the differentiablity of $x \mapsto \int_{a}^x W_{p}(x-y) (\omega(y) - p) \cH^{\omega}(y) \diff y$, for $x \in (a,\infty) \setminus \{a_1, \dots, a_{n+1}\}$. In view of Lemma~\ref{lemma:diff-volterra}, for what follows, we set $H(x) = (\omega(x) - p) \cH^{\omega}(x)$; {note that $H$ is indeed piecewise continuous on $(a,\infty)$.}


{
First, we consider $g(x)=W_p(x)$. Indeed, under Assumption~\ref{standing-assumption}, we have that $W_p$ is continuously differentiable on $(0,\infty)$ and its derivative is non-negative. It is well known that $W_p^\prime$ is integrable on $(0,A)$, for any $A>0$. Further, if $\sigma^2>0$ or, if $\sigma^2=0$ and $\nu(0,\infty)<\infty$, then $W_p^\prime(0+) \in (0,\infty)$. Otherwise, if $\sigma^2=0$ and $X$ has paths of UBV, or if $\sigma^2=0$ and $X$ has paths of BV with $\nu(0,\infty)=\infty$, then $W_p^\prime(0+)=\infty$. However, in these last two cases, since we further have that $W_p^\prime$ is convex on $(0,\infty)$, then Lemma~\ref{lemma:diff-volterra} can be applied. In particular, we have obtained: for $x \in \reals \setminus \{a_1, \dots, a_{n+1}\}$,
\begin{equation}\label{H:first:derivative}
\cH^{\omega \prime}(x) = Z_p^\prime(x-a;\Phi(\phi)) + W_p(0+) (\omega(x)-p) \cH^{\omega}(x) + \int_{a}^x  W^\prime_p(x-y) (\omega(y)-p) \cH^{\omega}(y) \diff y .
\end{equation}
Note that, if $X$ has paths of UBV, i.e., if $W_p(0+)=0$, then~\eqref{H:first:derivative} becomes
\begin{equation}\label{H:first:derivative-UBV}
\cH^{\omega \prime}(x) = Z_p^\prime(x-a;\Phi(\phi)) + \int_{a}^x  W^\prime_p(x-y) (\omega(y)-p) \cH^{\omega}(y) \diff y .
\end{equation}
}

{
Second, if $\sigma^2>0$, then we consider $g(x)=W_p^\prime(x)$. Recall that with $\sigma^2>0$, we have $W_p(0+)=0$, $W_p^\prime(0+)=2/\sigma^2$ and $W_p$ is twice continuously differentiable on $(0,\infty)$. Consequently, $W_p^\prime$ is continuously differentiable and, since it is also assumed to be convex on $(0,\infty)$, its derivative $W_p^{\prime \prime}$ is non-decreasing. Finally, since $W_p^\prime(0+)<\infty$ and $W_p^\prime$ is continuously differentiable, then $W_p^{\prime \prime}$ is integrable on $(0,A)$, for any $A>0$. In other words, Lemma~\ref{lemma:diff-volterra} can be applied with $g(x)=W_p^\prime(x)$ to the integral in~\eqref{H:first:derivative-UBV}. In particular, we have obtained: for $x \in \reals \setminus \{a_1, \dots, a_{n+1}\}$,
\begin{equation}\label{H:second:derivative}
\cH^{\omega \prime \prime}(x) = Z_p^{\prime \prime}(x-a;\Phi(\phi)) + \frac{2}{\sigma^2} (\omega(x)-p) \cH^{\omega}(x) + \int_{a}^x  W^{\prime \prime}_p(x-y) (\omega(y)-p) \cH^{\omega}(y) \diff y .
\end{equation}
}

\subsection{Convexity of the derivative}

The next proposition provides a crucial property for the analysis of our control problem.

\begin{proposition}\label{prop:log-convex-general}
If the tail of the Lévy measure is log-convex, i.e., if the function $x \mapsto \nu(x,\infty)$ is log-convex on $(0,\infty)$, then $\mathcal{H}^{\omega \prime}$ is log-convex on $(0,\infty)$. In particular, it is also convex on $(0,\infty)$.
\end{proposition}

\begin{proof}
Recall from~\eqref{H:first:derivative} that, for $p = \rho$, we have
{
\[
\cH^{\omega \prime}(x) = Z_\rho^\prime(x-a;\Phi(\phi)) + W_\rho(0+) (\omega(x)-\rho) \cH^{\omega}(x) + \int_{a}^x  W^\prime_\rho(x-y) (\omega(y)-\rho) \cH^{\omega}(y) \diff y .
\]
}
Then, for $x \in (0,\infty)$, we can write
\begin{equation}\label{Sc:Fn:Omega:Prime:Positive}
\cH^{\omega \prime}(x) = Z_{\rho}^\prime(x-a;\Phi(\phi)) + \int_{a}^0  W'_{\rho}(x-y)( \omega(y)-\rho )  \cH^{\omega}(y) \diff y .
\end{equation}

On the one hand, it follows from (the proof of) Proposition~2 in \cite{renaud2019} and from Theorem~1.2 in \cite{loeffen-renaud_2010}, that, under our assumption on the Lévy measure, the function $x\mapsto Z_{\rho}^\prime(x-a;\Phi(\phi))$ is log-convex on $(a,\infty)$ and the function $x \mapsto W'_{\rho}(x-y)$ is log-convex on $(y,\infty)$. Hence, using that Riemann integrals are limits of partial sums and using the fact that $y \mapsto ( \omega(y)-\rho ) \cH^{\omega}$ is non-negative, we conclude that $x \mapsto \int_{a}^0  W'_{\rho}(x-y)( \omega(y)-\rho )  \cH^{\omega}(y) \diff y$ is log-convex on $(0,\infty)$. Thus, $\cH^{\omega \prime}$ is log-convex on $(0,\infty)$.
%
\end{proof}



\section{Optimisation problem}\label{Sec:Optim:Dividend}

Now, we go back to the optimisation problem described in Section~\ref{sect:optimisation-problem}. Let us fix the discount rate $q >0$ and the bankruptcy rate function $\omega$, which is a discounting intensity with $\rho=0$.

As motivated in \cite{albrecher-et-al_2011}, the monotonicity of $\omega$ models the fact that, the deeper the process dives into the red zone, the higher the rate at which bankruptcy is declared increases. The extra assumption that it is a piecewise continuous function is arguably very mild from a modelling perspective, while it has the mathematical advantage of yielding analytical properties of the associated Omega scale functions that are in line with those for classical scale functions.

For simplicity, let us define $\homega(x) := q + \omega(x)$. Note that $\homega$ is a discounting intensity, as in Definition~\ref{def:bankruptcy-rate-function}, with $\rho=q$. In particular, we have that $\homega(x) = \phi_q$ for all $x \in (-\infty,a)$, where $\phi_q := q + \phi$.

\subsection{Performance function}

It is interesting to note that, as in \cite{renaud2019}, we have the following alternative representation of the performance function.
\begin{lemma}\label{lemma:another-representation}
For any $\pi \in \Pi$, we have
\begin{equation*}
v_\pi(x) = \e_x \left[ \int_0^\infty \mathrm{e}^{- \int_0^t \omega_q (U^\pi_s) \mathrm{d}s} \mathrm{d}L^\pi_t \right] , \quad x \in \reals.
\end{equation*}
\end{lemma}
\begin{proof}
By definition (see Equation~\eqref{omega-ruin}),
\[
\ruintime = \inf \left\lbrace t>0 \colon \int_0^t \omega(U^\pi_s) \mathrm{d}s > \mathbf{e}_1 \right\rbrace ,
\]
where $\mathbf{e}_1$ is an exponentially distributed random variable with unit mean that is independent of $\mathcal{F}_\infty$. Consequently, for $t>0$, we can write
\[
\left\lbrace \sigma_\omega^\pi > t \right\rbrace = \left\lbrace \int_0^t \omega(U^\pi_s) \mathrm{d}s \leq \mathbf{e}_1 \right\rbrace
\]
and hence
\[
\e_x \left[ \int_0^{\sigma_\omega^\pi} \mathrm{e}^{-q t} \mathrm{d}L^\pi_t \right] = \e_x \left[ \int_0^\infty \mathrm{e}^{-q t} \e_x \left[ \ind_{\{\sigma_\omega^\pi > t\}} \vert \mathcal{F}_\infty \right] \mathrm{d}L^\pi_t \right] = \e_x \left[ \int_0^\infty \mathrm{e}^{-q t - \int_0^t \omega(U^\pi_s) \mathrm{d}s} \mathrm{d}L^\pi_t \right] .
\]
\end{proof}

In other words, the current control problem is equivalent to a control problem with an infinite time horizon in which the dividend payments are penalized by a level-dependent discounting. See Section~2 in \cite{renaud2019} for more interpretations. In addition, the identity in Lemma~\ref{lemma:another-representation} will be used in the verification procedure of the control problem.

\subsection{Barrier strategies}\label{sect:barrier-strategies}

Of particular interest to our problem is the family of (horizontal) barrier strategies. For $b \in \reals$, the barrier strategy at level $b$ is the strategy denoted by $\pi^b$ and with cumulative amount of dividends paid until time $t$ given by $L_t^b = \left( \sup_{0<s\leq t} X_s - b \right)_+$, for $t>0$. If $X_0 = x > b$, then $L^b_0 = x-b$. Note that, if $b \geq 0$, then $\pi_b \in \Pi$. The corresponding value function is thus given by
\[
v_b (x) = \e_x \left[ \int_0^{\ruinbarrier} \mathrm{e}^{-q t} \mathrm{d}L_t^b \right] , \quad x \in \reals ,
\]
where $\ruinbarrier$ is the termination/bankruptcy time corresponding to the controlled process $U^b_t = X_t - L_t^b$.

First, let us compute the performance function of an arbitrary barrier strategy. Following the same line of reasoning as in the proof of Theorem~4.1 of \cite{czarna-et-al_2020} or the proof of Proposition 1 in \cite{renaud2019}, we have the following result. The details are left to the reader.
\begin{proposition}\label{prop:perf-barrier}
Fix $b \geq 0$. We have
\[
v_b (x) = \e_x \left[ \int_0^{\ruinbarrier} \mathrm{e}^{-q t} \mathrm{d}L_t^b \right] =
\begin{cases}
\frac{\cH^{\homega}(x)}{\cH^{\homega\prime}(b)} & \text{if $x \in (-\infty,b]$,} \\
x-b+ \frac{\cH^{\homega}(b)}{\cH^{\homega\prime}(b)} & \text{if $x \in [b, \infty)$,}
\end{cases}
\]
where $\cH^{\homega\prime}(0):=\cH^{\homega\prime}(0+)$.
\end{proposition}

Note that $\cH^{\homega \prime}(0+)$ can only be different from $\cH^{\homega \prime}(0-)$ when $X$ {is a compound Poisson process} and $\omega(0-)>0$. See Theorem~\ref{thm:differentiability}.

We propose the following definition for our candidate optimal barrier level:
\begin{equation}\label{eq:optimal-level}
b_\omega^\ast = \argmin_{b \geq 0} \cH^{\homega\prime}(b) .
\end{equation}

If the function $x \mapsto \nu(x,\infty)$ is log-convex on $(0,\infty)$, then $b_{\omega}^\ast$ given by~\eqref{eq:optimal-level} is well defined and finite. Indeed, if $x \mapsto \nu(x,\infty)$ is log-convex, then from Proposition \ref{prop:log-convex-general} we have that $\cH^{\homega \prime}$ is non-negative and convex on $(0,\infty)$. In addition, thanks to the representation \eqref{Sc:Fn:Omega:Prime:Positive} and the fact that $\lim_{x\to\infty} Z_q'(x-a;\Phi(q_\phi)) = +\infty$ (see the proof of Proposition 2 in \cite{renaud2019}) we deduce that $\lim_{x \to \infty}\cH^{\homega\prime}(x) = +\infty$. It follows that $\cH^{\homega\prime}$ is ultimately strictly increasing. This allows us to conclude that $b_{\omega}^\ast$ is well defined.

\subsection{Solution of the control problem}

Let us state our second main result, asserting the optimality of the barrier strategy at level $b_{\omega}^\ast$ among all admissible strategies.

\begin{theorem}\label{thm:main-result}
If the tail of the Lévy measure is log-convex, then the barrier strategy at level $b_\omega^\ast$, as given by~\eqref{eq:optimal-level}, is an optimal strategy for the control problem. Moreover, we have
\[
v_\ast (x) =
\begin{cases}
\frac{\cH^{\homega}(x)}{\cH^{\homega\prime}(b_\omega^\ast)} & \text{for $x \in (-\infty,b_\omega^\ast]$,} \\
x-b_\omega^\ast + \frac{\cH^{\homega}(b_\omega^\ast)}{\cH^{\homega\prime}(b_\omega^\ast)} & \text{for $x \in [b_\omega^\ast,\infty)$.}
\end{cases}
\]
\end{theorem}

\section{Proof of Theorem~\ref{thm:main-result}.}\label{Sec:Verification}

This section is devoted to the proof Theorem~\ref{thm:main-result}. {Recall that we assume the tail of the Lévy measure is log-convex.}

\subsection{Variational inequalities}

First, set
\begin{equation}\label{eq:generator}
\Gamma v(x) = \gamma v^\prime (x)+\frac{\sigma^2}{2} v^{\prime \prime}(x) + \int_{(0,\infty)} \left( v(x-z)-v(x)+v'(x)z\ind_{(0,1]}(z) \right) \nu(\mathrm{d}z) ,
\end{equation}
where $v \colon \reals \to \reals$ is a \textit{sufficiently smooth} function, i.e., with first and second derivatives defined almost everywhere. If a derivative of $v$ does not exist at $x$, then it is replaced in~\eqref{eq:generator} by its left-hand version. In what follows, we will use the following notation:
\[
\left(\Gamma - (q+\omega) \right)v(x) := \Gamma v(x) - (q+\omega(x)) v(x) ,
\]
for each $x$ where the quantities are well defined.


\begin{lemma}\label{lem_HJB}
For $x \in (0,\infty)$, we have $v^\prime_{b^\ast}(x) \geq 1$ and, for $x \in \reals \setminus \{a_1,\dots,a_{n+1}\}$, we have
\[
(\Gamma-(q+\omega)) v_{b^\ast}(x) \leq 0 .
\]
\end{lemma}
\begin{proof}
{Recall Proposition~\ref{prop:perf-barrier} and Theorem~\ref{thm:differentiability} to deduce the necessary smoothness properties of $v_{b_{\omega}^\ast}$ for what follows.} By the definition of a barrier strategy as well as the definition of $b_{\omega}^\ast$, we have that $v'_{b_{\omega}^\ast}(x) = 1$ for $x \in [b_{\omega}^\ast,\infty)$ and $v'_{b_{\omega}^\ast}(x) \leq 1$ for $x \in (0,b_{\omega}^\ast)$.

Now, we focus on the proof of the variational inequality {for $x \in (-\infty,b^\ast_\omega)\setminus\lbrace a_1,\cdots,a_{n+1} \rbrace$. Recall that $\homega(x) = q + \omega(x)$. First, by the Strong Markov property, we can write
\[
v_{b^\ast_{\omega}}(x) =  \e_{x} \left[ \mathrm e^{-\int_0^{\tau_{b^{\ast}_{\omega}}^+}\omega_q(X_s) \diff s } \right] v_{b^\ast_{\omega}}(b^{\ast}_{\omega}).
\]
Next, we observe that the process
\[
\mathrm e^{-\int_0^{t\wedge\tau_{b^{\ast}_{\omega}}^+} \omega_q(X_s) \diff s} \e_{X_{t\wedge\tau_{b^{\ast}_{\omega}}^+}} \left[ \mathrm e^{-\int_0^{\tau_{b^{\ast}_{\omega}}^+} \omega_q(X_s) \diff s } \right] , \quad t\geq 0,
\]
is a $\p_x$-martingale. Indeed, by the Markov property of $X$, we have
\begin{align*}
\e_x\left[ \mathrm e^{-\int_0^{\tau_{b^{\ast}_{\omega}}^+} \omega_q(X_s) \diff s } \Big| \mathcal{F}_t\right]
&= \Ind_{\lbrace t \geq \tau_{b^{\ast}_{\omega}}^+ \rbrace} \mathrm e^{-\int_0^{\tau_{b^{\ast}_{\omega}}^+} \omega_q(X_s) \diff s } + \Ind_{\lbrace t < \tau_{b^{\ast}_{\omega}}^+ \rbrace} \mathrm e^{-\int_0^{t}\omega_q(X_s) \diff s } \e_x\left[ \mathrm e^{-\int_t^{\tau_{b^{\ast}_{\omega}}^+} \omega_q(X_s) \diff s } \Big| \mathcal{F}_t\right] \\
&= \Ind_{\lbrace t \geq \tau_{b^{\ast}_{\omega}}^+ \rbrace} \mathrm e^{-\int_0^{\tau_{b^{\ast}_{\omega}}^+} \omega_q(X_s) \diff s } + \Ind_{\lbrace t < \tau_{b^{\ast}_{\omega}}^+ \rbrace} \mathrm e^{-\int_0^{t}\omega_q(X_s) \diff s } \e_{X_t}\left[ \mathrm e^{-\int_0^{\tau_{b^{\ast}_{\omega}}^+} \omega_q(X_s) \diff s } \right] \\
&= \mathrm e^{-\int_0^{t\wedge\tau_{b^{\ast}_{\omega}}^+} \omega_q(X_s) \diff s } \e_{X_{t\wedge\tau_{b^{\ast}_{\omega}}^+}} \left[ \mathrm e^{-\int_0^{\tau_{b^{\ast}_{\omega}}^+} \omega_q(X_s) \diff s } \right].
\end{align*}
It now follows that the right-hand side is a $\p_x$-martingale. In addition, we obtain that $\mathrm e^{-\int_0^{t\wedge\tau_{b^{\ast}_{\omega}}^+} \omega_q(X_s) \diff s } v_{b^\ast_{\omega}}(X_{t\wedge\tau_{b^{\ast}_{\omega}}^+})$ is a $\p_x$-martingale, from which we deduce that (see, e.g., \cite[Equation~(12)]{biffis-kyprianou_2010} or \cite[Lemma 5.7]{NOBA2021})
\[
(\Gamma - (q+\omega))v_{b^\ast_{\omega}}(x) = 0 , \quad x \in (-\infty,b^\ast_{\omega})\setminus\lbrace a_1,\cdots,a_{n+1} \rbrace .
\]
We remark that $x\mapsto (\Gamma - (q+\omega))v_{b^\ast_{\omega}}(x)$ is continuous on $(-\infty,b^\ast_{\omega})\setminus\lbrace a_1,\cdots,a_{n+1} \rbrace$, thanks to Theorem~\ref{thm:differentiability}.
}

Finally, the verification of the variational inequality on $[b_{\omega}^\ast, \infty)$ can be performed as in \cite{loeffen_2008} since, under our assumptions, $\cH^{\homega \prime}$ is log-convex, thanks to Proposition~\ref{prop:log-convex-general}. We leave the details to the reader.
\end{proof}

\subsection{Verification procedure}

Let $\pi \in \Pi$ be an arbitrary admissible strategy. For simplicity, let us write $h=v_{b_\omega^\ast}$, $L=L^\pi$ and $U=U^\pi$. {In Appendix~\ref{proof-of-BV-density}, we show that $h$ has a density that is of finite variation.} As a consequence, $h$ can be written as the difference of two convex functions (see, e.g., Theorem~A on p.\ 23 in \cite{roberts-varberg_1973}) and, by the Meyer-It\^{o} Formula (see Theorem~70 on p.\ 218 in \cite{protter_2004}), we can write
\begin{align*}
h (U_t) - h(U_0) &= \int_{0+}^t h^{\prime} (U_{s-}) \diff U_s + \frac12 \int_{-\infty}^\infty \ell_t^a \mu (\diff a) \\
& \qquad + \sum_{0 < s \leq t} \left\lbrace h (U_{s}) - h (U_{s-}) - h^{\prime} (U_{s-}) \Delta U_s \right\rbrace ,
\end{align*}
where $\mu (\diff a)$ is the weak second derivative of $h$ and $\ell_t^a$ is the semimartingale local time of $U$. From Corollary~1 on p.\ 219 in \cite{protter_2004} (see also, e.g., the proofs of Lemma~4.2 and Theorem~1.2 in \cite{kyprianou-et-al_2010}), it is known that, if $\sigma>0$, then by the occupation formula we have
\begin{align*}
\int_{-\infty}^\infty \ell_t^a \mu (\diff a) = \sigma^2 \int_{0}^t h''(U_s) \diff s .
\end{align*}
Otherwise, this integral is equal to zero.


%

Consequently, we can further write (after several standard manipulations)
\begin{align*}
h (U_t) - h (U_0) &= \int_{0+}^t \Gamma h (U_s) \diff s - \int_{0}^t h^{\prime} (U_s) \diff L^c_s \\
& \qquad - \sum_{0 < s \leq t} \left\lbrace h (X_s-L_{s-}) - h (X_s - L_{s-} - \Delta L_s) \right\rbrace + M_t ,
\end{align*}
where $L^c$ is the continuous part of $L$ and where $M$ is a local martingale. In this last expression, the quantity $\Gamma h (U_s)$ is well defined, thanks to Theorem~\ref{thm:differentiability}.

If we define the process $Y_t = \exp \left( - \int_0^t \homega(U_s) \diff s \right)$, which has paths of BV and is such that
\[
\mathrm{d}Y_t = - \homega(U_t) Y_t \diff t , \quad Y_0=1,
\]
then by the Integration by Parts Formula for semimartingales (see Corollary 2 on p.\ 68 in \cite{protter_2004}) we have
\begin{align*}
Y_t h(U_t) - h (U_0) &= \int_{0+}^t Y_s\left( \Gamma h (U_s) - (\omega(U_s)+q) h (U_s) \right) \diff s - \int_{0}^t Y_s h^{\prime} (U_s) \diff L^c_s \\
& \qquad - \sum_{0 < s \leq t} Y_s \left\lbrace h (X_s-L_{s-}) - h (X_s - L_{s-} - \Delta L_s) \right\rbrace + \int_{0+}^t Y_s \diff M_s ,
\end{align*}
where $t \mapsto \int_{0+}^t Y_s \diff M_s$ is also a local martingale.

Since by admissibility $\int_{(0,\infty)} \ind_{\lbrace U_t < 0 \rbrace} \diff L_t = 0$ and since, by Lemma~\ref{lem_HJB}, $h^\prime (x) \geq 1$ for all $x \in (0,\infty)$, then
\[
\int_{0}^t Y_s h^{\prime} (U_s) \diff L^c_s \geq \int_{0}^t Y_s \diff L^c_s ,
\]
and $h (X_s-L_{s-}) - h (X_s - L_{s-} - \Delta L_s) \geq \Delta L_s$, almost surely. Also, since for at least all $x \in \reals \setminus \{a_1,\dots,a_{n+1}\}$ we have
\[
(\Gamma-(q+\omega)) h(x) \leq 0 ,
\]
then we can further write
\begin{align*}
Y_t h(U_t) - h (U_0) &\leq - \int_{0}^t Y_s \diff L^c_s - \sum_{0 < s \leq t} Y_s \left\lbrace \Delta L_s \right\rbrace + \int_{0+}^t Y_s \diff M_s \\
&= - \int_{0}^t Y_s \diff L_s + \int_{0+}^t Y_s \diff M_s ,
\end{align*}
almost surely. Indeed, since $\int_{(0,\infty)} \ind_{\lbrace U_t < 0 \rbrace} \diff L_t = 0$, then $U$ behaves as the SNLP $X$ (which does not have monotone paths) below zero and it is well known that, for any $x \in \reals$, the set $\{t \geq 0 \colon X_t = x\}$ has zero Lebesgue measure.

Now, consider $(T_n)_{n\geq 1}$ a localising sequence. Applying the last inequality at $t=T_n$ and taking expectations, we can write
\[
h(x) \geq \e_x \left[ Y_{T_n} h(U_{T_n}) \right] + \e_x \left[ \int_{0}^{T_n} Y_s \diff L_s \right] .
\]
Taking the limits as $n \to \infty$, and since $h$ is non-negative, we get
\[
h(x) \geq \e_x \left[ \int_{0}^{\infty} Y_t \diff L_t \right] = \e_x \left[ \int_0^\infty \mathrm{e}^{-\int_0^t \homega(U_s) \diff s} \diff L_t \right] = v_\pi(x) ,
\]
where the last equality is taken from Lemma~\ref{lemma:another-representation}. Written differently, we have obtained that $h(x) \equiv v_{b^\ast_\omega}(x) \geq v_\pi(x)$, for all $x \in \reals$. As $\pi$ is arbitrary, the result follows.

\section{Sensitivity analysis with respect to the bankruptcy rate function}\label{Sec:Numerical}

Intuitively, if $\omega$ is large, then the penalisation for being in the \textit{red zone} is large (see Lemma~\ref{lemma:another-representation}), which means that bankruptcy will occur sooner than if $\omega$ were otherwise smaller; in this case, as with classical ruin, the process should stay away from zero.  Conversely, if $\omega$ is small, then bankruptcy will occur later. Consequently, in that case, it is less risky to enter the red zone, so it is expected that $b_\omega^*$ will be (can afford to be) closer to zero. In conclusion, it is expected that $b_\omega^*$ will be monotonic increasing with respect to $\omega$.

In this section, we will illustrate this by providing numerical illustrations of the impact of the bankruptcy rate function on the value function and on the optimal barrier level. In what follows, we assume that
\[
X_t = X_0 + \mu t + \sigma B_t - \sum_{i = 1}^{N_t} Y_i ,
\]
where $\mu = 0.075$ and $\sigma = 0.25$, $B=\{B_t, t\geq 0\}$ is a standard Brownian motion, $\{N_t, t\geq 0\}$ is a homogeneous Poisson process with arrival rate $\lambda = 0.5$ (independent of $B$) and $(Y_i)_{i \geq 1}$ is a sequence of independent exponentially distributed random variables with parameter $\alpha = 9$. Note that this model falls under our assumptions of a Lévy measure being absolutely continuous with a density that is log-convex (which implies that the tail is log-convex). In other words, Theorem~\ref{thm:main-result} is valid for this model.

In this model, the corresponding classical scale function $W_q$ can be computed explicitly as 
\[
W_q(x) = \sum_{i=1}^3 D_i \mathrm{e}^{\theta_i x},
\]
where $\theta_i, \, i = 1,2,3$ are the distinct roots of $s \mapsto \psi(s) - q$ satisfying $\theta_1 > 0$ and $\theta_2,\, \theta_3 < 0$, and where $D_i = \frac{1}{\psi'(\theta_i)}$. Recall that the scale function $Z_q$ can then be computed explicitly by simple integration; see Equation~\eqref{eq:Zq-def}.

In what follows, to help with comparisons, each bankruptcy rate function will be such that $\phi = 1.5$ and $a = -1$, according to the notation introduced in Definition~\ref{def:bankruptcy-rate-function}. As $W_q$ and $Z_q$ are known explicitly, we will be able to compute values of the corresponding Omega scale functions using {Picard iterations (see, e.g., \cite[(46)--(49)]{czarna-et-al_2019})}.

Also, throughout this analysis, we will compare our results with the case of Parisian ruin at rate $\phi$, i.e., corresponding to $\omega_{\textsc{P}}(x) := \phi \ind_{(-\infty,0)} (x)$.

Finally, let us fix $q = 0.05$.

\subsection{Step functions}

First, let us consider step functions as in Equation~\eqref{omega-step-function}. For $n=1,2, 3, 4, 5$, set
\begin{equation}\label{eq:step-functions}
\omega_n(x) := \phi \ind_{(-\infty,a)}(x) + p_1 \ind_{[a,a_2)}(x) + \dots + p_n \ind_{[a_n, 0)}(x) ,
\end{equation}
with $a_i = a/i$ and $p_i=\phi/(i+1)$. Recall that $a=-1$ and $\phi=1.5$. Note that, if $n_1<n_2$, then $\omega_{n_2}(x) < \omega_{n_1}(x)$ for all $x<0$. Also, let us define $\omega_0:=\omega_{\textsc{P}}$. See Figure~\ref{Fig:Omega_Sequence_Step}. 

\begin{figure}[h!]
	\centering
	\includegraphics[scale=1]{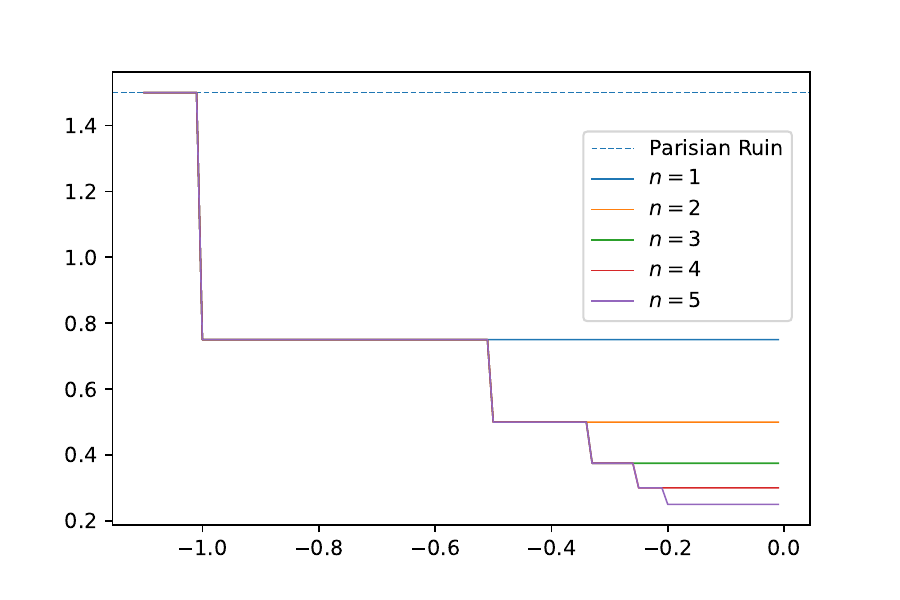}
	\caption{Bankruptcy rate functions defined in Equation~\eqref{eq:step-functions}.}
	\label{Fig:Omega_Sequence_Step}
\end{figure}

In Figure~\ref{Fig:Omega:Step}, we plot the value functions corresponding to the bankruptcy rate functions $\omega_n$ defined above (in Equation~\eqref{eq:step-functions}), with the corresponding optimal barrier levels $b_n^*$. Note that, as expected, we have that if $n_1<n_2$, then $b_{n_2}^* \leq b_{n_1}^*$. See also Figure~\ref{Fig:b_omega_nsteps}.
%

%
%
\begin{figure}[h!]
	\centering
	\includegraphics[scale=1]{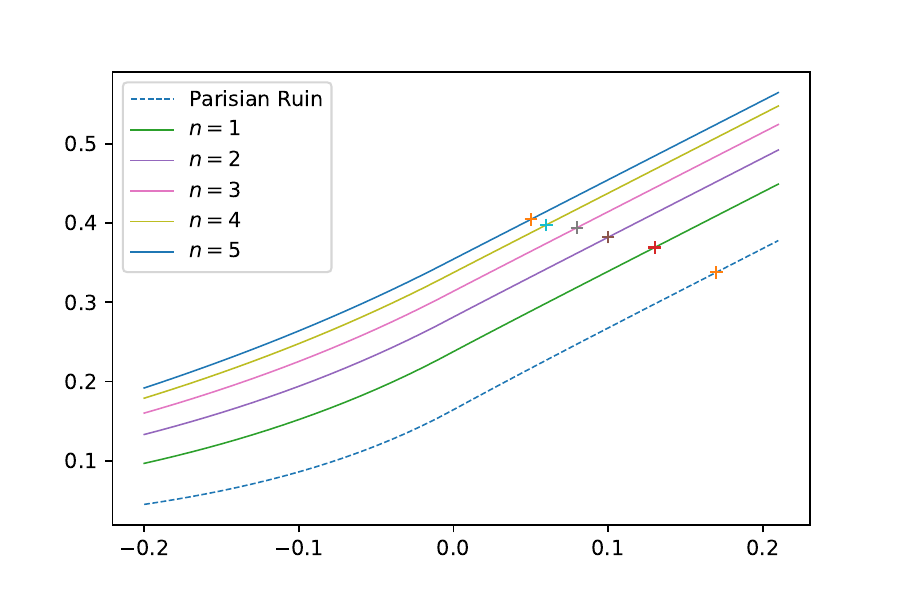}
	\caption{Value functions corresponding to the step bankruptcy rate functions defined in Equation~\eqref{eq:step-functions}, with corresponding optimal barrier levels indicated by a cross.}
	\label{Fig:Omega:Step}
\end{figure}

\begin{figure}[h!]
	\centering
	\includegraphics[scale=1]{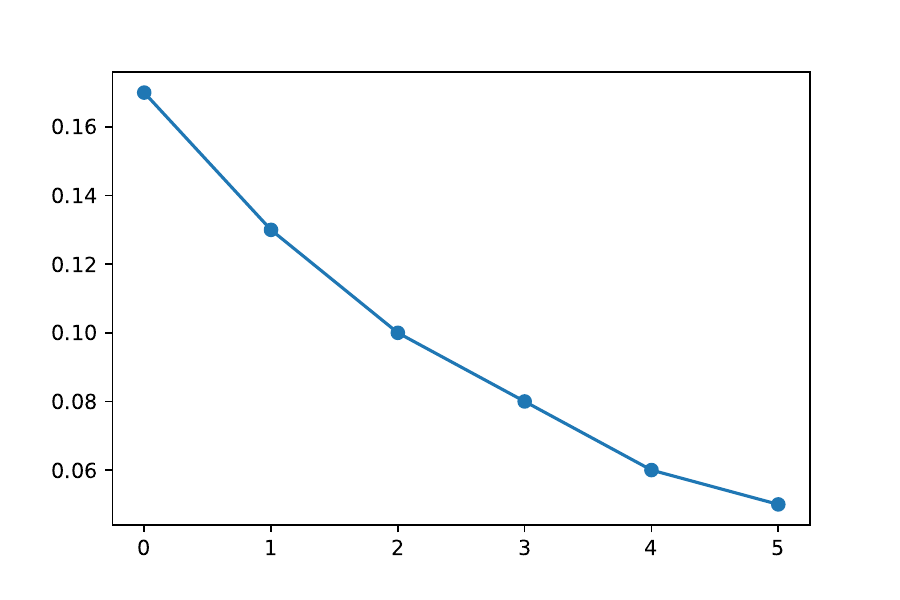}
	\caption{Optimal barrier levels corresponding to the step bankruptcy rate functions defined in Equation~\eqref{eq:step-functions}.}
	\label{Fig:b_omega_nsteps}
\end{figure}

\subsection{Affine functions}

Now, we consider affine bankruptcy rate functions, i.e., of the form
\begin{equation}\label{eq:affine-functions}
\omega_m(x) := \phi \ind_{(-\infty,a)} (x) + (\phi + m(x-a)) \ind_{[a,0)} (x) ,
\end{equation}
for different values of $m < 0$. Recall that $a=-1$ and $\phi=1.5$. Note that, if $m=0$, then $\omega_0=\omega_{\textsc{P}}$, i.e., it corresponds to Parisian ruin at rate $\phi$. Finally, if $m_1<m_2$, then $\omega_{m_1}(x) < \omega_{m_2}(x)$ for all $x \in (a,0)$. See Figure~\ref{Fig:Omega_Sequence_Affine}. 

\begin{figure}[h!]
	\centering
	\includegraphics[scale=1]{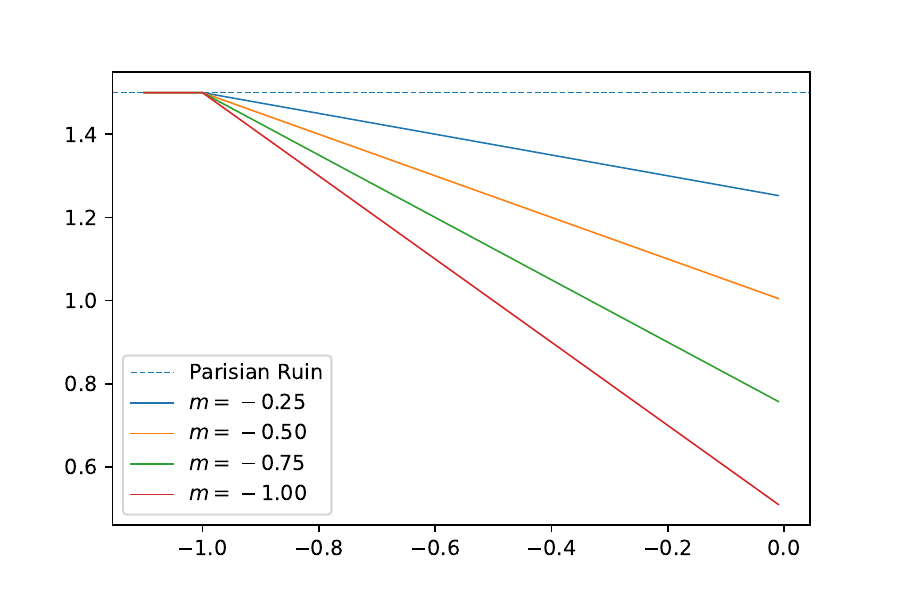}
	\caption{Bankruptcy rate functions defined in Equation~\eqref{eq:affine-functions}.}
	\label{Fig:Omega_Sequence_Affine}
\end{figure}

In Figure~\ref{Fig:Omega:Affine}, we plot the value functions corresponding to the bankruptcy rate functions $\omega_m$ defined above (in Equation~\eqref{eq:affine-functions}), with the corresponding optimal barrier levels $b_m^*$. Note that, as expected, we have that if $m_1<m_2$, then $b_{m_2}^* \leq b_{m_1}^*$. See also Figure~\ref{Fig:Affine:Barrier}.

\begin{figure}[h!]
	\centering
		\includegraphics[scale=1]{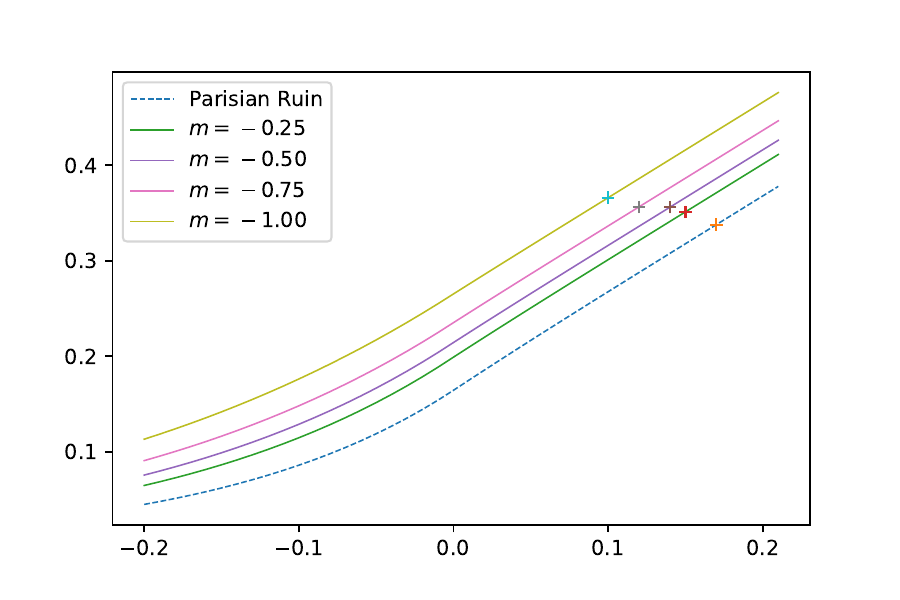}
	\caption{Value functions corresponding to the affine bankruptcy rate functions defined in Equation~\eqref{eq:affine-functions}, with corresponding optimal barrier levels indicated by a cross.}
	\label{Fig:Omega:Affine}
\end{figure}

\begin{figure}[h!]
	\centering
	\includegraphics[scale=1]{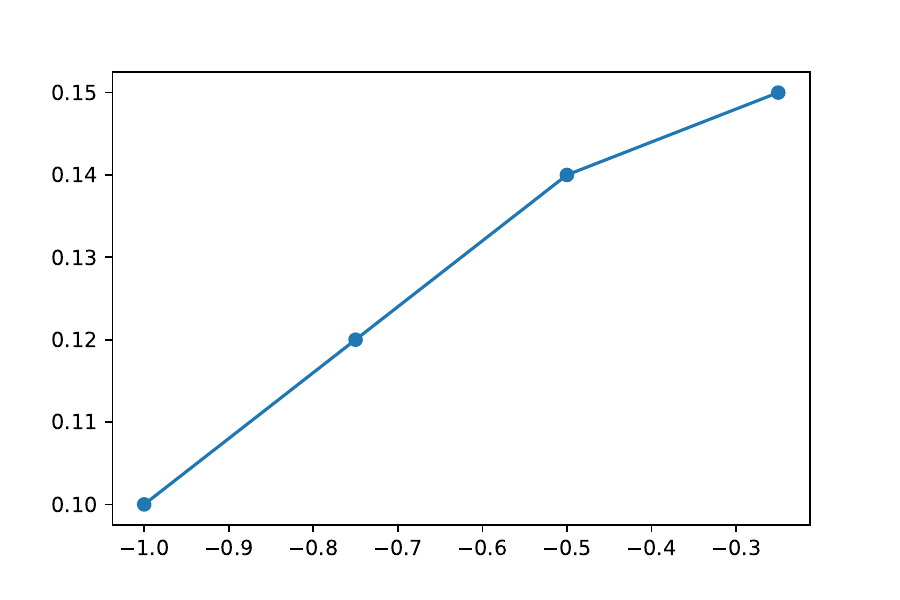}
	\caption{Optimal barrier levels corresponding to the affine bankruptcy rate functions defined in Equation~\eqref{eq:affine-functions}.}
	\label{Fig:Affine:Barrier}
\end{figure}

\section*{Acknowledgements}

We would like to express our gratitude to an anonymous reviewer for their patience and for their detailed reading of our manuscript during this lengthy reviewing process. Thanks to their helpful comments, significant improvements have been made.

Funding in support of this work was provided by a CRM-ISM Postdoctoral Fellowship from the Centre de recherches mathématiques (CRM) and the Institut des sciences mathématiques (ISM), by a FRQNT Postdoctoral Fellowship (369449) from the Fonds de Recherche du Qu\'ebec, and by Discovery Grants (RGPIN-2019-06538,RGPIN-2025-05758) from the Natural Sciences and Engineering Research Council of Canada (NSERC).

%
%
\bibliographystyle{abbrv}
\bibliography{references-SNLPs.bib, references_de-finetti.bib}

\appendix

\section{Proof of Lemma~\ref{lemma:identities}}\label{A:identities}

Our proof of~\eqref{eq:identity-Z} is based on the representation of $Z_r (\cdot;\Phi(s))$ given in~\eqref{eq:Zq-def} and on~\eqref{eq:identity-W}.  Take $x>0$ and use~\eqref{eq:Zq-def} to write
\begin{align*}
\int_0^x W_p(x-y) Z_r (y;\Phi(s)) \mathrm{d}y &= \int_0^x W_p(x-y) \left[ (s-r) \int_0^\infty \mathrm{e}^{-\Phi(s)z} W_r(y+z) \mathrm{d}z \right] \mathrm{d}y \\
&= (s-r) \int_0^\infty \mathrm{e}^{-\Phi(s)z} \left[ \int_0^x  W_p(x-y) W_r(y+z) \mathrm{d}y \right] \mathrm{d}z .
\end{align*}
Using~\eqref{eq:identity-W} to compute the inner integral and then using~\eqref{eq:Zq-def} again, we can write
\begin{multline*}
\int_0^x W_p(x-y) Z_r (y;\Phi(s)) \mathrm{d}y = \frac{s-r}{p-r} \left[ \frac{Z_p(x;\Phi(s))}{s-p} - \frac{Z_r(x;\Phi(s))}{s-r}\right] \\
- (s-r) \int_0^\infty \mathrm{e}^{-\Phi(s)z} \left[ \int_0^z W_p(x+z-y) W_r (y) \mathrm{d}y \right] \mathrm{d}z ,
\end{multline*}
where
\begin{align*}
\int_0^\infty \mathrm{e}^{-\Phi(s)z} & \left[ \int_0^z W_p(x+z-y) W_r (y) \mathrm{d}y \right] \mathrm{d}z \\
&= \int_0^\infty W_r (y) \left[ \int_y^\infty \mathrm{e}^{-\Phi(s)z} W_p(x+z-y)  \mathrm{d}z \right] \mathrm{d}y \\
&= \frac{1}{s-p} Z_p (x; \Phi(s)) \int_0^\infty \mathrm{e}^{-\Phi(s)y} W_r (y) \mathrm{d}y \\
&= \frac{1}{(s-p)(s-r)} Z_p (x; \Phi(s)) .
\end{align*}
In conclusion,
\[
\int_0^x W_p(x-y) Z_r (y;\Phi(s)) \mathrm{d}y = \frac{Z_p(x;\Phi(s)) - Z_r(x;\Phi(s))}{p-r} .
\]
The result follows.

\section{The (candidate) value function has a density of finite variation}\label{proof-of-BV-density}

In this appendix, we show that $v_{b_\omega^\ast}$ has a density that is of finite variation. Before proceeding, recall from Proposition~\ref{prop:perf-barrier} that
\[
v_{b_\omega^\ast} (x) =
\begin{cases}
\frac{\cH^{\homega}(x)}{\cH^{\homega\prime}(b_\omega^\ast)} & \text{for $x \in (-\infty,b_\omega^\ast]$,} \\
x-b_\omega^\ast + \frac{\cH^{\homega}(b_\omega^\ast)}{\cH^{\homega\prime}(b_\omega^\ast)} & \text{for $x \in [b_\omega^\ast,\infty)$.}
\end{cases}
\]

First, recall from Lemma~\ref{lem:alternate-functional-equations} (see also Equation~\eqref{eq:Zq-def}) with $p =\phi_q$ and $\omega=\omega_q$ that we have
\[
\cH^{\homega}(x) = \mathrm e^{\Phi(\phi_q)(x-a)} + \int_{a}^x  W_{\phi_q}(x-y) (\omega(y)-\phi) \cH^{\homega}(y) \diff y, \quad x \in \reals .
\]
Then, from equation~\eqref{H:first:derivative}, we deduce that
\[
\cH^{\homega\prime}(x) = \Phi(\phi_q) \mathrm e^{\Phi(\phi_q)(x-a)} + W_{\phi_q}(0+) (\omega(x)-\phi) \cH^{\homega}(x) + \int_{a}^x  W^\prime_{\phi_q}(x-y) (\omega(y)-\phi) \cH^{\homega}(y) \diff y .
\]
For $x \in (-\infty,a)$, we observe that $\cH^{\homega\prime}(x) = \Phi(\phi_q) \mathrm e^{\Phi(\phi_q)(x-a)}$ which is increasing. Hence, $\cH^{\homega\prime}$ is of finite variation on $(-\infty,a)$. On $(0,\infty)$, thanks to Proposition~\ref{prop:log-convex-general}, we have that $\cH^{\homega\prime}$ is convex, hence it is of finite variation on any compact of $(0,\infty)$.

Next, recall that $\lbrace a_1,\cdots,a_{n+1} \rbrace$ is a partition of $[a,0]$; in particular, $a_1=a$ and $a_{n+1}=0$. Let us fix $i \in \lbrace 1,\cdots,n \rbrace$ and show that $\cH^{\homega\prime}$ is of finite variation on $(a_i,a_{i+1})$.
%

Since $\omega$ is non-increasing and $\cH^{\homega}$ increasing, then $x \mapsto (\omega(x)-\phi) \cH^{\homega}(x)$ is of finite variation on $(a_i,a_{i+1})$. Next, since $W_{\phi_q}^\prime (z) = 0$ for $z<0$, we can write
\[
\int_{a}^x  W^\prime_{\phi_q}(x-y) (\omega(y)-\phi) \cH^{\homega}(y) \diff y = \int_{a}^{\infty}  W^\prime_{\phi_q}(x-y) (\omega(y)-\phi) \cH^{\homega}(y) \diff y .
\]
Let $\lbrace x_1, \cdots, x_{N+1} \rbrace$ be a partition of $(a_i,a_{i+1})$, with $x_1=a_i$ and $x_{N+1}=a_{i+1}$. We have
\begin{multline*}
\sum_{j=1}^{N} \Big| \int_{a}^{\infty}  W^\prime_{\phi_q}(x_{j+1}-y) (\omega(y)-\phi) \cH^{\homega}(y) \diff y  - \int_{a}^{\infty}  W^\prime_{\phi_q}(x_j-y) (\omega(y)-\phi) \cH^{\homega}(y) \diff y \Big| \\
\leq \int_a^{\infty} \sum_{j=1}^{N} | W^\prime_{\phi_q}(x_{j+1}-y) - W^\prime_{\phi_q}(x_j-y) | |\omega(y)-\phi| \cH^{\homega}(y) \diff y.
\end{multline*}

Fix $y>a$. Note that if $z \in (a_i, a_{i+1})$ then $z-y>0$. Since $x \mapsto W^\prime_{\phi_q}(x-y)$ is convex on $(a_i, a_{i+1})$ (thanks to Proposition~\ref{prop:log-convex-general}), it follows that there exists $\hat{a} \in [a_i,a_{i+1}]$ such that $W^\prime_{\phi_q}(\hat{a}-y) \leq W^\prime_{\phi_q}(x-y)$ for all $x \in [a_i,a_{i+1}]$. Hence, we can write
\[
\sum_{j=1}^{N} | W^\prime_{\phi_q}(x_{j+1}-y) - W^\prime_{\phi_q}(x_j-y) | = W^\prime_{\phi_q}(a_i-y) + W^\prime_{\phi_q}(a_{i+1}-y) - 2 W^\prime_{\phi_q}(\hat{a}-y) .
\]

Consequently, we deduce that
\begin{align*}
\sum_{j=1}^{N} & \Big| \int_{a}^{\infty}  W^\prime_{\phi_q}(x_{j+1}-y) (\omega(y)-\phi) \cH^{\homega}(y) \diff y  - \int_{a}^{\infty}  W^\prime_{\phi_q}(x_j-y) (\omega(y)-\phi) \cH^{\homega}(y) \diff y \Big| \\
\leq & \int_{a}^{\infty}  W^\prime_{\phi_q}(a_i-y) |\omega(y)-\phi| \cH^{\homega}(y) \diff y + \int_{a}^{\infty}  W^\prime_{\phi_q}(a_{i+1}-y) |\omega(y)-\phi| \cH^{\homega}(y) \diff y \\
& \qquad - 2 \int_{a}^{\infty}  W^\prime_{\phi_q}(\hat{a}-y) |\omega(y)-\phi| \cH^{\homega}(y) \diff y .
\end{align*}
Note that this last quantity is finite and independent of the chosen partition.

In conclusion, we have that $\cH^{\homega\prime}$ is locally of finite variation on $\reals$.

Finally, since $v_{b_{\omega}^\ast}$ is affine on $(b_{\omega}^\ast,\infty)$, hence its derivative has finite variation on $(b_{\omega}^\ast,\infty)$. Also, as $v_{b_{\omega}^\ast}$ is proportional to $\cH^{\homega}$ on $(-\infty,b_{\omega}^\ast)$, then $v'_{b_{\omega}^\ast}$ is (locally) of finite variation on $(-\infty,b_{\omega}^\ast)$. The result follows. 
\end{document}